\documentclass[12pt]{article}

\usepackage{amsmath, epsfig, cite}
\usepackage{amssymb}
\usepackage{amsfonts}
\usepackage{latexsym}
\usepackage{amsthm}

\newtheorem{thm}{Theorem}[section]

\newtheorem{cor}[thm]{Corollary}

\newtheorem{lem}[thm]{Lemma}

\newcommand{\binq}[2]{\genfrac{[}{]}{0mm}{0}{#1}{#2}}
\newcommand{\tbnq}[2]{\genfrac{[}{]}{0mm}{1}{#1}{#2}}
\numberwithin{equation}{section}

\renewcommand{\thefootnote}{*}

\begin{document}

\begin{center}
{\large\bf $q$-Supercongruences with parameters\footnote{This work is supported by the National Natural Science Foundations of China (Nos. 12071103 and
11661032).}}
\end{center}

\renewcommand{\thefootnote}{$\dagger$}

\vskip 2mm \centerline{Chuanan Wei}
\begin{center}
{School of Biomedical Information and Engineering\\ Hainan Medical University, Haikou 571199, China\\
{\tt weichuanan78@163.com  } }
\end{center}


\vskip 0.7cm \noindent{\bf Abstract.} In terms of the creative
microscoping method recently introduced by Guo and Zudilin  [Adv.
Math. 346 (2019), 329--358], we find a $q$-supercongruence with four
parameters modulo $\Phi_n(q)(1-aq^n)(a-q^n)$,  where $\Phi_n(q)$
denotes the $n$-th cyclotomic polynomial in $q$. Then we empoly it and the Chinese remainder theorem
for coprime polynomials to derive a $q$-supercongruence
with two parameters modulo $[n]\Phi_n(q)^3$, where $[n]=(1-q^n)/(1-q)$ is
the $q$-integer.



\section{Introduction}

 For any complex numbers $x$ and $q$, define the $q$-shifted factorial to be
 \begin{equation*}
(x;q)_{0}=1\quad\text{and}\quad
(x;q)_n=(1-x)(1-xq)\cdots(1-xq^{n-1})\quad \text{when}\quad
n\in\mathbb{N}.
 \end{equation*}
For shortening some formulas in this paper, we also adopt the
notation
\begin{equation*}
(x_1,x_2,\dots,x_m;q)_{n}=(x_1;q)_{n}(x_2;q)_{n}\cdots(x_m;q)_{n}.
 \end{equation*}
Following Gasper and Rahman \cite{Gasper}, define the basic
hypergeometric series $_{s+1}\phi_{s}$ by
$$
_{s+1}\phi_{s}\left[\begin{array}{c}
a_1,a_2,\ldots,a_{s+1}\\
b_1,b_2,\ldots,b_{s}
\end{array};q,\, z
\right] =\sum_{k=0}^{\infty}\frac{(a_1,a_2,\ldots, a_{s+1};q)_k}
{(q,b_1,b_2,\ldots,b_{s};q)_k}z^k.
$$
Then Watson's $_8\phi_7$ transformation (cf. \cite[Appendix
(III.18)]{Gasper}) can be expressed as
\begin{align}
& _{8}\phi_{7}\!\left[\begin{array}{cccccccc}
a,& qa^{\frac{1}{2}},& -qa^{\frac{1}{2}}, & b,    & c,    & d,    & e,    & q^{-n} \\
  & a^{\frac{1}{2}}, & -a^{\frac{1}{2}},  & aq/b, & aq/c, & aq/d, & aq/e, & aq^{n+1}
\end{array};q,\, \frac{a^2q^{n+2}}{bcde}
\right] \notag\\[5pt]
&\quad =\frac{(aq, aq/de;q)_{n}} {(aq/d, aq/e;q)_{n}}
\,{}_{4}\phi_{3}\!\left[\begin{array}{c}
aq/bc,\ d,\ e,\ q^{-n} \\
aq/b,\, aq/c,\, deq^{-n}/a
\end{array};q,\, q
\right]. \label{eq:watson}
\end{align}

Recently, Guo and Schlosser \cite[Theorems 5.1 and 5.3]{GS}) proved
that, for any odd integer $n>1$,
\begin{align}
&\sum_{k=0}^{M}(-1)^k[4k-1]\frac{(q^{-1};q^2)_k^5}{(q^2;q^2)_k^5}q^{k^2+5k}
\notag\\[5pt]
&\equiv[n](-q)^{(n+1)(n-3)/4}\sum_{k=0}^{(n+1)/2}\frac{(q^{-1};q^2)_k^2(q^3;q^2)_k}{(q^2;q^2)_k^3}q^{3k}\hspace*{-1.5mm}\pmod{[n]\Phi_n(q)^2},
\label{eq:guo-a}\\[5pt]
&\sum_{k=0}^{M}[4k-1]\frac{(q^{-1};q^2)_k^4}{(q^2;q^2)_k^4}q^{4k}
\equiv0\pmod{[n]\Phi_n(q)^2}. \label{eq:guo-b}
\end{align}
Here and throughout the paper, $M$ is always equal to $(n+1)/2$ or $(n-1)$.
Two generalizations of \eqref{eq:guo-b} due to Guo and Schlosser
\cite[Theorems 1.1 and 1.2]{GS19} can be stated as
\begin{align}
&\sum_{k=0}^{M}[4k-1]\frac{(q^{-1};q^2)_k^4}{(q^2;q^2)_k^4}q^{4k}
\equiv-(1+3q+q^2)[n]^4\pmod{[n]^4\Phi_n(q)},
\notag\\[5pt]
&\sum_{k=0}^{M}[4k-1]\frac{(aq^{-1},q^{-1}/a;q^2)_k(q^{-1};q^2)_k^2}{(q^{2}/a,aq^2;q^2)_k(q^2;q^2)_k^2}q^{4k}
\equiv0\pmod{[n]^2(1-aq^n)(a-q^n)},
 \label{eq:guo-c}
\end{align}
where $n>1$ is an arbitrary odd integer. For some
$q$-supercongrueces with parameters, the reader may refer to Wei
\cite{Wei}. More recent $q$-supercongruences can be found in the
papers
\cite{Guo-rima,Guo-rama,Guo-jmaa,Guo-a2,Guopreprint,GuoZu,GuoZu2,GS20c,LP,NP,NP2,Tauraso,WY,Zu19}.

Inspired by the work just mentioned, we shall establish the
following two theorems.

\begin{thm}\label{thm-a}
Let $n>1$ be an odd integer. Then, modulo
$\Phi_n(q)(1-aq^n)(a-q^n)$,
\begin{align}
&\sum_{k=0}^{M}[4k-1]\frac{(aq^{-1},q^{-1}/a,q^{-1}/b,cq^{-1},dq^{-1},q^{-1};q^2)_k}{(q^{2}/a,aq^2,bq^2,q^2/c,q^{2}/d,q^2;q^2)_k}\bigg(\frac{bq^7}{cd}\bigg)^k
\notag\\[5pt]\:
&\:\equiv[n](bq)^{(n+1)/2}\frac{(q^{-2}/b;q^2)_{(n+1)/2}}{(bq^2;q^2)_{(n+1)/2}}
\sum_{k=0}^{(n+1)/2}\frac{(aq^{-1},q^{-1}/a,q^{-1}/b,q^3/cd;q^2)_k}{(q^2,q^{-2}/b,q^2/c,q^2/d;q^2)_k}q^{2k}.
\label{eq:wei-a}
\end{align}
\end{thm}

\begin{thm}\label{thm-b}
Let $n>1$ be an odd integer. Then, modulo $[n]\Phi_n(q)^3$,
\begin{align}
&\sum_{k=0}^{M}[4k-1]\frac{(q^{-1};q^2)_k^4(cq^{-1},dq^{-1};q^2)_k}{(q^2;q^2)_k^4(q^2/c,q^2/d;q^2)_k}\bigg(\frac{q^7}{cd}\bigg)^k
\notag\\[5pt]\:
&\:\equiv\frac{\Omega_q(n)(1-q)^2(q^3/cd;q^2)_2}{(1+q)^2(q^2/c,q^2/d;q^2)_2}
\sum_{k=0}^{(n-3)/2}\frac{(q^3;q^2)_k^3(q^7/cd;q^2)_k}{(q^2,q^6,q^6/c,q^6/d;q^2)_k}q^{2k},
\label{eq:wei-b}
\end{align}
where
\begin{align*}
\Omega_q(n)=[n]^3\bigg\{\frac{n^2(1-q)^2-(1+22q+q^2)}{24}-\frac{1}{q[n]^2[n-1][n+1]}\bigg\}\frac{q^{(n+5)/2}}{1+q^2}.
\end{align*}
\end{thm}

For two nonnegative integer $s,t$ with $s\leq t$, it is well known
that the $q$-binomial coefficient $\tbnq{t}{s}$ is a polynomial in
$q$ and
\begin{align}\label{polynomial}
\frac{(q;q^2)_t}{(q^2;q^2)_t}=\frac{1}{(-q;q)_t^2}\binq{2t}{t}.
\end{align}

Letting $a\to1, b\to0, c\to1, d\to1$ in Theorem \ref{thm-a} and using \eqref{polynomial} and Lemma \ref{lem-b}, we
obtain \eqref{eq:guo-a} immediately. Fixing $cd=q^3$ in  Theorem
\ref{thm-a} and utilizing Lemma \ref{lem-b}, we get the following generalization of
\eqref{eq:guo-b}:
\begin{align*}
&\sum_{k=0}^{M}[4k-1]\frac{(aq^{-1},q^{-1}/a,q^{-1}/b,q^{-1};q^2)_k}{(q^{2}/a,aq^2,bq^2,q^2;q^2)_k}(bq^4)^k
\\[5pt]\:
&\:\equiv[n](bq)^{(n+1)/2}\frac{(q^{-2}/b;q^2)_{(n+1)/2}}{(bq^2;q^2)_{(n+1)/2}}\pmod{[n](1-aq^n)(a-q^n)},
\end{align*}
where $n>1$ is an arbitrary odd integer.

 It is easy to understand that the factor $(q^2;q^2)_k^4(q^2/c,q^2/d;q^2)_k$
 in the denominator of the left-hand side of  \eqref{eq:wei-b} is
 relatively prime to $\Phi_n(q)$ as $c\to1,d\to \infty$. Some similar arguements will be omitted in the rest of the paper.
On the basis of the above discussions, the  $c\to1,d\to \infty$ case
of Theorem \ref{thm-b} yields the following $q$-supercongruence.

\begin{cor}\label{corl-a}
Let $n>1$ be an odd integer. Then, modulo $[n]\Phi_n(q)^3$,
\begin{align*}
\sum_{k=0}^{M}(-1)^k[4k-1]\frac{(q^{-1};q^2)_k^5}{(q^2;q^2)_k^5}q^{k^2+5k}
\equiv\frac{\Omega_q(n)}{[2]^3[4]}\sum_{k=0}^{(n-3)/2}\frac{(q^3;q^2)_k^3}{(q^2;q^2)_k(q^6;q^2)_k^2}q^{2k}.
\end{align*}
\end{cor}

Choosing $M=(p^r+1)/2$ and letting $q\to 1$ in Corollary
\ref{corl-a}, we achieve the supercongruence: for $p>3$,
\begin{align*}
\sum_{k=0}^{(p^r+1)/2}(-1)^k(4k-1)\frac{(-\frac{1}{2})_k^5}{k!^5}
\equiv
\frac{p^r(p^{2r}-p^{4r}-1)}{16(p^{2r}-1)}\sum_{k=0}^{(p^r-3)/2}\frac{(\frac{3}{2})_k^3}{k!(k+2)!^2}\pmod{p^{r+3}},
\end{align*}
where $r$ is a positive integer, $p$ is an odd prime, and the two
notations will frequently be employed in this section.

The  $c\to q^{-2},d\to \infty$ case of Theorem \ref{thm-b} produces
the following formula.

\begin{cor}\label{corl-b}
Let $n>3$ be an odd integer. Then, modulo $[n]\Phi_n(q)^3$,
\begin{align*}
\sum_{k=0}^{M}(-1)^k[4k-1]\frac{(q^{-1};q^2)_k^4(q^{-3};q^2)_k}{(q^2;q^2)_k^4(q^{4};q^2)_k}q^{k^2+7k}
\equiv\frac{\Omega_q(n)}{[2]^2[4][6]}\sum_{k=0}^{(n-3)/2}\frac{(q^3;q^2)_k^3}{(q^2,q^6,q^8;q^2)_k}q^{2k}.
\end{align*}
\end{cor}

Setting $M=(p^r+1)/2$ and letting $q\to 1$ in Corollary
\ref{corl-b}, we are led to
\begin{align*}
&\hspace{-5mm}\sum_{k=0}^{(p^r+1)/2}(-1)^k(4k-1)\frac{(-\frac{1}{2})_k^4(-\frac{3}{2})_k}{k!^4(k+1)!}\\[5pt]
&\hspace{-5mm}\quad\equiv
\frac{p^r(p^{2r}-p^{4r}-1)}{16(p^{2r}-1)}\sum_{k=0}^{(p^r-3)/2}\frac{(\frac{3}{2})_k^3}{k!(k+2)!(k+3)!}\pmod{p^{r+3}}
\quad\text{with}\quad p>3.
\end{align*}

The  $c\to1,d\to1$ case of Theorem \ref{thm-b} gives the following
conclusion.

\begin{cor}\label{corl-c}
Let $n>1$ be an odd integer. Then, modulo $[n]\Phi_n(q)^3$,
\begin{align*}
\sum_{k=0}^{M}[4k-1]\frac{(q^{-1};q^2)_k^6}{(q^2;q^2)_k^6}q^{7k}
\equiv\frac{\Omega_q(n)[3][5]}{[2]^4[4]^2}\sum_{k=0}^{(n-3)/2}\frac{(q^3;q^2)_k^3(q^7;q^2)_k}{(q^2;q^2)_k(q^6;q^2)_k^3}q^{2k}.
\end{align*}
\end{cor}

Taking $M=(p^r+1)/2$ and letting $q\to 1$ in Corollary \ref{corl-c},
we obtain
\begin{align*}
\sum_{k=0}^{(p^r+1)/2}(4k-1)\frac{(-\frac{1}{2})_k^6}{k!^6} \equiv
\frac{15p^r(p^{2r}-p^{4r}-1)}{64(p^{2r}-1)}\sum_{k=0}^{(p^r-3)/2}\frac{(\frac{3}{2})_k^3(\frac{7}{2})_k}{k!(k+2)!^3}\pmod{p^{r+3}}.
\end{align*}

The $c\to1,d\to q^{-2}$ case of Theorem \ref{thm-b} yields the
following relation.

\begin{cor}\label{corl-d}
Let $n>3$ be a positive odd integer. Then, modulo $[n]\Phi_n(q)^3$,
\begin{align*}
\sum_{k=0}^{M}[4k-1]\frac{(q^{-1};q^2)_k^5(q^{-3};q^2)_k}{(q^2;q^2)_k^5(q^{4};q^2)_k}q^{9k}
\equiv\frac{\Omega_q(n)[5][7]}{[2]^3[4]^2[6]}\sum_{k=0}^{(n-3)/2}\frac{(q^3;q^2)_k^3(q^9;q^2)_k}{(q^2,q^6,q^6,q^8;q^2)_k}q^{2k}.
\end{align*}
\end{cor}

Choosing $M=(p^r+1)/2$ and letting $q\to 1$ in Corollary
\ref{corl-d}, we get
\begin{align*}
&\sum_{k=0}^{(p^r+1)/2}(4k-1)\frac{(-\frac{1}{2})_k^5(-\frac{3}{2})_k}{k!^5(k+1)!}\\[5pt]
&\quad\equiv
\frac{35p^r(p^{2r}-p^{4r}-1)}{64(p^{2r}-1)}\sum_{k=0}^{(p^r-3)/2}\frac{(\frac{3}{2})_k^3(\frac{9}{2})_k}{k!(k+2)!^2(k+3)!}\pmod{p^{r+3}}\quad\text{for}\quad
p>3.
\end{align*}

The $c\to q^{-2},d\to q^{-2}$ case of Theorem \ref{thm-b} produces
 the following result.

\begin{cor}\label{corl-e}
Let $n>3$ be an odd integer. Then, modulo $[n]\Phi_n(q)^3$,
\begin{align*}
\sum_{k=0}^{M}[4k-1]\frac{(q^{-1};q^2)_k^4(q^{-3};q^2)_k^2}{(q^2;q^2)_k^4(q^{4};q^2)_k^2}q^{11k}
\equiv\frac{\Omega_q(n)[7][9]}{[2]^2[4]^2[6]^2}\sum_{k=0}^{(n-3)/2}\frac{(q^3;q^2)_k^3(q^{11};q^2)_k}{(q^2,q^6,q^8,q^8;q^2)_k}q^{2k}.
\end{align*}
\end{cor}

Setting $M=(p^r+1)/2$ and letting $q\to 1$ in Corollary
\ref{corl-e}, we gain
\begin{align*}
&\sum_{k=0}^{(p^r+1)/2}(4k-1)\frac{(-\frac{1}{2})_k^4(-\frac{3}{2})_k^2}{k!^4(k+1)!^2}\\[5pt]
&\quad\equiv
\frac{63p^r(p^{2r}-p^{4r}-1)}{64(p^{2r}-1)}\sum_{k=0}^{(p^r-3)/2}\frac{(\frac{3}{2})_k^3(\frac{11}{2})_k}{k!(k+2)!(k+3)!^2}\pmod{p^{r+3}}
\quad\text{with}\quad p>3.
\end{align*}

The rest of the paper is arranged as follows. Via the
 creative microscoping method, a $q$-supercongruence with four parameters modulo $\Phi_n(q)(1-aq^n)(a-q^n)$ will be
 established in Section 2. Then we employ it and the Chinese remainder theorem
for coprime polynomials to deduce a $q$-supercongruence with four
parameters modulo $\Phi_n(q)(1-aq^n)(a-q^n)(b-q^n)$ in Section 3, and
finally prove Theorem \ref{thm-b} with the help of this parametric
$q$-supercongruence.
\section{Proof of Theorem \ref{thm-a}}

 For proving Theorem \ref{thm-a}, we need the following
 lemma.

\begin{lem}\label{lem-a}
Let $n>1$ be an odd integer. Then
\begin{align}\label{eq:wei-c}
\sum_{k=0}^{M}[4k-1]\frac{(aq^{-1},q^{-1}/a,q^{-1}/b,cq^{-1},dq^{-1},q^{-1};q^2)_k}{(q^2/a,aq^2,bq^2,q^2/c,q^2/d,q^2;q^2)_k}\bigg(\frac{bq^7}{cd}\bigg)^k
\equiv0\pmod{\Phi_n(q)}.
\end{align}
\end{lem}

\begin{proof}
Let $\beta_q(k)$ be the $k$-th term on the left-hand side of
\eqref{eq:wei-c}, i.e.,
\begin{align*}
\beta_q(k)=[4k-1]\frac{(aq^{-1},q^{-1}/a,q^{-1}/b,cq^{-1},dq^{-1},q^{-1};q^2)_k}{(q^2/a,aq^2,bq^2,q^2/c,q^2/d,q^2;q^2)_k}\bigg(\frac{bq^7}{cd}\bigg)^k.
\end{align*}
By means of the known formula (cf. \cite[Formula (5.4)]{GS}):
\begin{align*}
\frac{(xq^{-1};q^2)_{(n+1)/2-k}}{(q^2/x;q^2)_{(n+1)/2-k}}\equiv(-x)^{(n+1)/2-2k}\frac{(xq^{-1};q^2)_{k}}{(q^2/x;q^2)_{k}}q^{(n-1)^2/4+3k-1}\pmod{\Phi_n(q)},
\end{align*}
we obtain
\begin{align*}
\beta_{q}((n+1)/2-k)\equiv-\beta_{q}(k)\pmod{\Phi_n(q)}.
\end{align*}
When $(n+1)/2$ is odd, it is routine to verify that
\begin{align}\label{eq:wei-d}
\sum_{k=0}^{(n+1)/2}[4k-1]\frac{(aq^{-1},q^{-1}/a,q^{-1}/b,cq^{-1},dq^{-1},q^{-1};q^2)_k}{(q^2/a,aq^2,bq^2,q^2/c,q^2/d,q^2;q^2)_k}\bigg(\frac{bq^7}{cd}\bigg)^k
\equiv0\pmod{\Phi_n(q)}.
\end{align}
When $(n+1)/2$ is even, the central term $\beta_{q}((n+1)/4)$ will
remain. Since it has the factor $[n]$, \eqref{eq:wei-d} is also true
in this instance. If $k>(n+1)/2$, the factor $(1-q^n)$ appears in
the numerator of $\beta_q(k)$. This indicates
\begin{align*}
\sum_{k=0}^{n-1}[4k-1]\frac{(aq^{-1},q^{-1}/a,q^{-1}/b,cq^{-1},dq^{-1},q^{-1};q^2)_k}{(q^2/a,aq^2,bq^2,q^2/c,q^2/d,q^2;q^2)_k}\bigg(\frac{bq^7}{cd}\bigg)^k
\equiv0\pmod{\Phi_n(q)}.
\end{align*}
Thus we complete the proof of Lemma \ref{lem-a}.
\end{proof}

\begin{proof}[Proof of Theorem \ref{thm-a}]
When $a=q^{-n}$ or $a=q^n$, the left-hand side of  \eqref{eq:wei-a}
is equal to
\begin{align}
&\sum_{k=0}^{M}[4k-1]\frac{(q^{-1-n},q^{-1+n},q^{-1}/b,cq^{-1},dq^{-1},q^{-1};q^2)_k}{(q^{2+n},q^{2-n},bq^2,q^2/c,q^2/d,q^2;q^2)_k}\bigg(\frac{bq^7}{cd}\bigg)^k
\notag\\[5pt]\:
&\:= \frac{-1}{q}{_8\phi_7}\!\left[\begin{array}{cccccccc} q^{-1},&
q^{\frac{3}{2}},& -q^{\frac{3}{2}},  & cq^{-1},    & dq^{-1}, &
q^{-1}/b, & q^{-1+n}, & q^{-1-n}
\\[5pt]
  & q^{-\frac{1}{2}}, & -q^{-\frac{1}{2}}, & q^2/c, & q^2/d, & bq^2,  & q^{2-n}, & q^{2+n}
\end{array};q^2,\, \frac{bq^7}{cd}
\right]. \label{watson-a}
\end{align}
Through \eqref{eq:watson}, the right-hand side of \eqref{watson-a}
can be rewritten as
\begin{align*}
[n](bq)^{(n+1)/2}\frac{(q^{-2}/b;q^2)_{(n+1)/2}}{(bq^2;q^2)_{(n+1)/2}}
 {_4\phi_3}\!\left[\begin{array}{cccccccc}
q^3/cd, &q^{-1}/b, &q^{-1+n}, &q^{-1-n}
\\[5pt]
  &q^2/c,&q^2/d, &q^{-2}/b
\end{array};q^2,\, q^2\right].
\end{align*}
This proves that the $q$-supercongruence  \eqref{eq:wei-a} holds
modulo $(1-aq^n)$ or $(a-q^n)$. Lemma \ref{lem-a} implies that
\eqref{eq:wei-a} is also true modulo $\Phi_n(q)$. Since $\Phi_n(q)$, $(1-aq^n)$,
and $(a-q^n)$ are pairwise relatively prime polynomials, we arrive
at \eqref{eq:wei-a}.
\end{proof}

\section{Proof of Theorem \ref{thm-b}}
In this section, we shall establish a $q$-supercongruence with four
parameters modulo $\Phi_n(q)(1-aq^n)(a-q^n)(b-q^n)$ by using Theorem
\ref{thm-a} and the Chinese remainder theorem for coprime
polynomials. Then we utilize it to provide a proof of
Theorem\ref{thm-b}.

\begin{thm}\label{thm-c}
Let $n>1$ be an odd integer. Then, modulo
$\Phi_n(q)(1-aq^n)(a-q^n)(b-q^n)$,
\begin{align*}
&\sum_{k=0}^{M}[4k-1]\frac{(aq^{-1},q^{-1}/a,q^{-1}/b,cq^{-1},dq^{-1},q^{-1};q^2)_k}{(q^{2}/a,aq^2,bq^2,q^2/c,q^{2}/d,q^2;q^2)_k}\bigg(\frac{bq^7}{cd}\bigg)^k
\\[5pt]\:
&\:\equiv[n]\Theta_q(a,b,n)
\sum_{k=0}^{(n+1)/2}\frac{(aq^{-1},q^{-1}/a,q^{-1}/b,q^3/cd;q^2)_k}{(q^2,q^{-2}/b,q^2/c,q^2/d;q^2)_k}q^{2k},
\end{align*}
where
\begin{align*}
\Theta_q(a,b,n)&=\frac{(b-q^n)(ab-1-a^2+aq^n)}{(a-b)(1-ab)}\frac{(bq)^{(n+1)/2}(q^{-2}/b;q^2)_{(n+1)/2}}{(bq^2;q^2)_{(n+1)/2}}\\[5pt]
&+\frac{(1-aq^n)(a-q^n)}{(a-b)(1-ab)}\frac{(b;q^2)_2(q^{-1};q^2)_{(n+1)/2}^2}{(q^{-1};q^2)_2(q^2/a,aq^{2};q^2)_{(n+1)/2}}.
\end{align*}
\end{thm}

\begin{proof}
Performing the replacements $n\to(n+1)/2, q\to q^2, a\to q^{-1},
b\to cq^{-1}, c\to dq^{-1}, d\to aq^{-1}, e\to q^{-1}/a$ in
\eqref{eq:watson}, we have
\begin{align*}
&\sum_{k=0}^{M}[4k-1]\frac{(aq^{-1},q^{-1}/a,q^{-1-n},cq^{-1},dq^{-1},q^{-1};q^2)_k}{(q^{2}/a,aq^2,q^{2+n},q^2/c,q^{2}/d,q^2;q^2)_k}\bigg(\frac{q^{7+n}}{cd}\bigg)^k
\\[5pt]\:
&\:=[n]\frac{(q^n;q^2)_2(q^{-1};q^2)_{(n+1)/2}^2}{(q^{-1};q^2)_2(q^2/a,aq^{2};q^2)_{(n+1)/2}}
\sum_{k=0}^{(n+1)/2}\frac{(aq^{-1},q^{-1}/a,q^{-1-n},q^3/cd;q^2)_k}{(q^2,q^{-2-n},q^2/c,q^2/d;q^2)_k}q^{2k}.
\end{align*}
Hence we achieve the following formula: modulo $(b-q^n)$,
\begin{align}
&\sum_{k=0}^{M}[4k-1]\frac{(aq^{-1},q^{-1}/a,q^{-1}/b,cq^{-1},dq^{-1},q^{-1};q^2)_k}{(q^{2}/a,aq^2,bq^2,q^2/c,q^{2}/d,q^2;q^2)_k}\bigg(\frac{bq^7}{cd}\bigg)^k
\notag\\[5pt]\:
&\:\equiv[n]\frac{(b;q^2)_2(q^{-1};q^2)_{(n+1)/2}^2}{(q^{-1};q^2)_2(q^2/a,aq^{2};q^2)_{(n+1)/2}}
\sum_{k=0}^{(n+1)/2}\frac{(aq^{-1},q^{-1}/a,q^{-1}/b,q^3/cd;q^2)_k}{(q^2,q^{-2}/b,q^2/c,q^2/d;q^2)_k}q^{2k}.
\label{eq:wei-e}
\end{align}
It is clear that the polynomials $(1-aq^n)(a-q^n)$ and $(b-q^n)$ are
relatively prime. Noting the relations
\begin{align*}
&\frac{(b-q^n)(ab-1-a^2+aq^n)}{(a-b)(1-ab)}\equiv1\pmod{(1-aq^n)(a-q^n)},
\\[5pt]
&\qquad\qquad\frac{(1-aq^n)(a-q^n)}{(a-b)(1-ab)}\equiv1\pmod{(b-q^n)}
\end{align*}
and employing the Chinese remainder theorem for coprime polynomials,
we can derive Theorem \ref{thm-c} from Theorem \ref{thm-a},
\eqref{eq:wei-e}, and Lemma \ref{lem-a}.
\end{proof}

In order to prove Theorem \ref{thm-b}, we also require the following lemma.

\begin{lem}\label{lem-b}
Let $n>1$ be an odd integer. Then
\begin{align*}
\sum_{k=0}^{M}[4k-1]\frac{(aq^{-1},q^{-1}/a,q^{-1}/b,cq^{-1},dq^{-1},q^{-1};q^2)_k}{(q^2/a,aq^2,bq^2,q^2/c,q^2/d,q^2;q^2)_k}\bigg(\frac{bq^7}{cd}\bigg)^k
\equiv0\pmod{[n]}.
\end{align*}
\end{lem}

\begin{proof}
 Let $\zeta\neq1$ be an $n$-th root of unity, which is not
necessarily primitive. This means that $\zeta$ is a primitive root
of unity of odd degree $t|n$. Lemma \ref{lem-a} with $n=t$ implies
that
\begin{align*}
\sum_{k=0}^{t-1}\beta_{\zeta}(k)=\sum_{k=0}^{(t+1)/2}\beta_{\zeta}(k)=0.
\end{align*}
According to the relation:
\begin{align*}
\frac{\beta_{\zeta}(jt+k)}{\beta_{\zeta}(jt)}=\lim_{q\to\zeta}\frac{\beta_{q}(jt+k)}{\beta_{q}(jt)}=\beta_{\zeta}(k),
\end{align*}
we get the following conclusions:
\begin{align*}
\sum_{k=0}^{n-1}\beta_{\zeta}(k)=\sum_{j=0}^{n/t-1}\sum_{k=0}^{t-1}\beta_{\zeta}(jt+k)
=\sum_{j=0}^{n/t-1}\beta_{\zeta}(jt)\sum_{k=0}^{t-1}\beta_{\zeta}(k)=0,
\end{align*}
\begin{align*}
\sum_{k=0}^{(n+1)/2}\beta_{\zeta}(k)=
\sum_{j=0}^{(n/t-3)/2}\beta_{\zeta}(jt)\sum_{k=0}^{t-1}\beta_{\zeta}(k)+\sum_{k=0}^{(t+1)/2}\beta_{\zeta}((n-t)/2+k)=0.
\end{align*}
They show that $\sum_{k=0}^{n-1}\beta_{q}(k)$ and
$\sum_{k=0}^{(n+1)/2}\beta_{q}(k)$ are both divisible by the
cyclotomic polynomials $\Phi_t(q)$. Because this is correct for any
divisor $t>1$ of $n$, we determine that the last two series are
divisible by
\begin{equation*}
\prod_{t|n,t>1}\Phi_t(q)=[n].
\end{equation*}
Therefore, we finish the proof of Lemma \ref{lem-b}.
\end{proof}

When $cd=q^3$, Theorem \ref{thm-c} becomes the result under Lemma \ref{lem-b}: modulo
$[n](1-aq^n)(a-q^n)(b-q^n)$,
\begin{align*}
&\sum_{k=0}^{M}[4k-1]\frac{(aq^{-1},q^{-1}/a,q^{-1}/b,q^{-1};q^2)_k}{(q^{2}/a,aq^2,bq^2,q^2;q^2)_k}(bq^4)^k
\\[5pt]
&\equiv[n]\frac{(b-q^n)(ab-1-a^2+aq^n)}{(a-b)(1-ab)}\frac{(bq)^{(n+1)/2}(q^{-2}/b;q^2)_{(n+1)/2}}{(bq^2;q^2)_{(n+1)/2}}
\\[5pt]
&\quad+[n]\frac{(1-aq^n)(a-q^n)}{(a-b)(1-ab)}\frac{(b;q^2)_2(q^{-1};q^2)_{(n+1)/2}^2}{(q^{-1};q^2)_2(q^2/a,aq^{2};q^2)_{(n+1)/2}}.
\end{align*}
The $b\to 1$ case of it implies that \eqref{eq:guo-c} is true modulo
$[n]\Phi_n(q)(1-aq^n)(a-q^n)$.

\begin{proof}[Proof of Theorem \ref{thm-b}]
Theorem \ref{thm-c} can be manipulated as follows:
 modulo
$\Phi_n(q)(1-aq^n)(a-q^n)(b-q^n)$,
\begin{align*}
&\sum_{k=0}^{M}[4k-1]\frac{(aq^{-1},q^{-1}/a,q^{-1}/b,cq^{-1},dq^{-1},q^{-1};q^2)_k}{(q^{2}/a,aq^2,bq^2,q^2/c,q^{2}/d,q^2;q^2)_k}\bigg(\frac{bq^7}{cd}\bigg)^k
\\[5pt]\:
&\:\equiv[n]\Theta_q(a,b,n)\sum_{k=0}^{1}\frac{(aq^{-1},q^{-1}/a,q^{-1}/b,q^3/cd;q^2)_k}{(q^2,q^{-2}/b,q^2/c,q^2/d;q^2)_k}q^{2k}\\[5pt]
&\:\quad+[n]\Theta_q(a,b,n)\sum_{k=0}^{(n-3)/2}\frac{(aq^{-1},q^{-1}/a,q^{-1}/b,q^3/cd;q^2)_{k+2}}{(q^2,q^{-2}/b,q^2/c,q^2/d;q^2)_{k+2}}q^{2k+4}.
\end{align*}
Letting $b\to1$  and applying the formula
\[(1-q^n)(1+a^2-a-aq^n)=(1-a)^2+(1-aq^n)(a-q^n),\]
we are led to the conclusion: modulo $\Phi_n(q)^2(1-aq^n)(a-q^n)$,
\begin{align}
&\sum_{k=0}^{M}[4k-1]\frac{(aq^{-1},q^{-1}/a,cq^{-1},dq^{-1};q^2)_k(q^{-1};q^2)_k^2}{(q^{2}/a,aq^2,q^2/c,q^{2}/d;q^2)_k(q^2;q^2)_k^2}\bigg(\frac{q^7}{cd}\bigg)^k
\notag\\[5pt]
&\:\equiv[n]R_q(a,n)\frac{(aq^{-1},q^{-1}/a,q^3/cd;q^2)_{2}}{(1+q^2)(1+q)^2(q^2/c,q^2/d;q^2)_{2}}
\sum_{k=0}^{(n-3)/2}\frac{(q^3,aq^{3},q^{3}/a,q^7/cd;q^2)_{k}}{(q^2,q^6,q^6/c,q^6/d;q^2)_{k}}q^{2k},
\label{eq:wei-f}
\end{align}
where
\begin{align*}
R_q(a,n)&=\bigg\{\frac{q^5(q^{-1};q^2)_{(n+1)/2}^2}{(q^{-1};q^2)_2(q^2/a,aq^{2};q^2)_{(n+1)/2}}-\frac{q^{(n+7)/2}}{(q^{n-1};q^2)_2}\bigg\}\\[5pt]
&\quad\times\frac{(1-aq^n)(a-q^n)}{(1-a)^2}-\frac{q^{(n+7)/2}}{(q^{n-1};q^2)_2}.
\end{align*}

Noting $q^n\equiv1\pmod{\Phi_n(q)}$ and two known relations due to
Guo \cite[Lemma 2.1]{Guopreprint}:
\begin{align*}
&(aq^2,q^2/a;q^2)_{(n-1)/2}
\equiv(-1)^{(n-1)/2}\frac{(1-a^n)q^{-(n-1)^2/4}}{(1-a)a^{(n-1)/2}}\pmod{\Phi_n(q)},
\\[5pt]
&(aq,q/a;q^2)_{(n-1)/2}
\equiv(-1)^{(n-1)/2}\frac{(1-a^n)q^{(1-n^2)/4}}{(1-a)a^{(n-1)/2}}\pmod{\Phi_n(q)},
\end{align*}
we can proceed as follows:
\begin{align}
&\frac{q^5(q^{-1};q^2)_{(n+1)/2}^2}{(q^{-1};q^2)_2(q^2/a,aq^{2};q^2)_{(n+1)/2}}-\frac{q^{(n+7)/2}}{(q^{n-1};q^2)_2}
\notag\\[5pt]
&\quad\equiv\frac{q^{(n+9)/2}}{(1-q)^2}-\frac{q^{(n+9)/2}n(1-a)a^{(n-1)/2}}{(1-q/a)(1-aq)(1-a^n)}\pmod{\Phi_n(q)}.
\label{eq:wei-g}
\end{align}
The combination of \eqref{eq:wei-f} and \eqref{eq:wei-g} gives that,
modulo $\Phi_n(q)^2(1-aq^n)(a-q^n)$,
\begin{align}
&\sum_{k=0}^{M}[4k-1]\frac{(aq^{-1},q^{-1}/a,cq^{-1},dq^{-1};q^2)_k(q^{-1};q^2)_k^2}{(q^{2}/a,aq^2,q^2/c,q^{2}/d;q^2)_k(q^2;q^2)_k^2}\bigg(\frac{q^7}{cd}\bigg)^k
\notag\\[5pt]
&\:\equiv[n]S_q(a,n)\frac{q^{(n+9)/2}(aq^{-1},q^{-1}/a,q^3/cd;q^2)_{2}}{(1+q^2)(1+q)^2(q^2/c,q^2/d;q^2)_{2}}
\sum_{k=0}^{(n-3)/2}\frac{(q^3,aq^{3},q^{3}/a,q^7/cd;q^2)_{k}}{(q^2,q^6,q^6/c,q^6/d;q^2)_{k}}q^{2k},
\label{eq:wei-h}
\end{align}
where
\begin{align*}
S_q(a,n)&=\bigg\{\frac{1}{(1-q)^2}-\frac{n(1-a)a^{(n-1)/2}}{(1-q/a)(1-aq)(1-a^n)}\bigg\}\\[5pt]
&\quad\times\frac{(1-aq^n)(a-q^n)}{(1-a)^2}-\frac{1}{q(q^{n-1};q^2)_2}.
\end{align*}

By the L' H\^{o}spital rule, we have
\begin{align*}
&\lim_{a\to1}\bigg\{\frac{1}{(1-q)^2}-\frac{n(1-a)a^{(n-1)/2}}{(1-q/a)(1-aq)(1-a^n)}\bigg\}\frac{(1-aq^n)(a-q^n)}{(1-a)^2}\\[5pt]
&=\lim_{a\to1}\frac{\{(1-q/a)(1-aq)(1-a^n)-n(1-a)a^{(n-1)/2}(1-q)^2\}(1-aq^n)(a-q^n)}{(1-q)^2(1-q/a)(1-aq)(1-a^n)(1-a)^2}\\[5pt]
&=[n]^2\frac{n^2(1-q)^2-(1+22q+q^2)}{24(1-q)^2}.
\end{align*}
Letting $a\to1$ in \eqref{eq:wei-h} and employing the above limit,
we arrive at Theorem \ref{thm-b} under \eqref{polynomial} and Lemma \ref{lem-b}.
\end{proof}

\end{document}